\documentclass[12pt]{amsart}
\newtheorem{thm}{Theorem}[section]
\newtheorem{cor}[thm]{Corollary}
\newtheorem{lem}[thm]{Lemma}

\theoremstyle{definition}

\theoremstyle{remark}

\numberwithin{equation}{section}

\begin{document}
\title[Hyers-Ulam stability]
{Hyers-Ulam stability of the spherical functions}%
\author[   Bouikhalene and
Eloqrachi] {\textsf{ Belaid Bouikhalene and Eloqrachi
Elhoucien } }%
\begin{abstract} In \cite{bbb} the authors obtained  the
Hyers-Ulam stability of the  functional equation
$$ \int_{K}\int_{ G} f(xtk\cdot y)d\mu(t)dk=f(x)g(y), \; x, y
\in G ,$$ where $G$ is a Hausdorff locally compact topological
group, $K$ is a copmact subgroup of morphisms of $G$,
  $\mu$ is a $K$-invariant complex measure with compact support, provided that the continuous function $f$ satisfies some Kannappan Type condition.
  The purpose of this
paper is to remove this restriction.\end{abstract} \maketitle
 \section{\bf Introduction}
 The stability problem of functional equations was posed
for the first time  by S. M. Ulam \cite{24} in the year 1940. Ulam
stated the problem as follows:
\begin{quote} Given a group $G_{1}$, a metric group ($G_{2},d$), a
number $\varepsilon>0$ and a mapping $f$: $G_{1}\longrightarrow
G_{2}$ which satisfies the inequality
$d(f(xy),f(x)f(y))<\varepsilon$ for all $x,y\in G_{1}$, does there
exist an homomorphism $h$: $G_{1}\longrightarrow G_{2}$ and a
constand $k>0$, depending only on $G_{1}$ and $G_{2}$ such that
$d(f(x),h(x))\leq k\varepsilon$ for all $x$ in $G_{1}$?\end{quote}
 The first affirmative answer was given by D. H. Hyers \cite{8}, under the assumption that $G_{1}$ and $G_{2}$ are
Banach spaces. \\ In 1978, Th. M. Rassias \cite{16} gave a
remarkable generalization of the Hyers's result which allows the
Cauchy difference to be unbounded, as follows:\begin{thm} \lbrack16]
Let $f : V\longrightarrow X$ be a mapping between Banach spaces and
let
 $p<1$ be fixed. If $f$ satisfies the inequality
$$\|f(x+y)-f(x)-f(y)\|\leq \theta(\|x\|^{p}+\|x\|^{p})$$ for some $\theta\geq 0$ and for all $x, y
\in V$ ($x, y \in V\setminus \{0\}$ if $p<0$). Then there exists a
unique additive  mapping $T : V\longrightarrow X$ such that
$$\|f(x)-T(x)\|\leq \frac{2\theta}{|2-2^{p}|}\|x\|^{p}$$ for all $x\in
V$ ($x\in V\setminus \{0\}$ if $p<0$).\\ If, in addition, $f(tx)$ is
continuous in $t$ for each fixed $x$, then $T$ is linear.\end{thm}
Several papers have been published in this subject and some
 interesting variants of Ulam's problem have been also
 investigated by a number of mathematicians. We refer the reader
to the following references \cite{2}, \cite{cad1}, \cite{5},
\cite{6}, \cite{7}, \cite{9}-\cite{20}.\\
  The stability of functional equations highlighted a new phenomenon
which is now usually called superstability. Consider the functional
equation $ E(f ) = 0$ and assume we are in a framework where the
notion of boundedness of $f$ and of $E( f)$ makes sense. We say that
the equation $ E(f ) = 0$ is superstable if the boundedness of $E(f
)$ implies that either $f$ is bounded or $f$  is a solution of $ E(f
) = 0$. This property was first
  observed when the following theorem was proved by J. Baker, J.
  Lawrence, and F. Zorzitto \cite{j1}\begin{thm}Let $V$ ba a vector space. If a function $f$: $V\longrightarrow \mathbb{R}$
  satisfies the inequality$$\mid f(x+y)-f(x)f(y)\mid\leq \varepsilon$$
  for some $\varepsilon>0$ and for all $x,y\in V$, then either $f$ is bounded on $V$ or $f(x+y)=f(x)f(y)$ for all $x,y\in V.$ \end{thm}
The result was generalized by J. A. Baker \cite{j2}, by replacing
$V$ by a semigroup and $\mathbb{R}$ by a normed algebra $E$, in
which the norm is multiplicative, i.e. $\|uv\| =\| u\|\|  v\| $, for
all $u, v \in E$, by R. Ger, P. \v{S}emrl \cite{se}, where $E$ is an
arbitrary commutative complex semisimple Banach algebra and by J.
Lawrence \cite{la} in the case where $E$ is the algebra of all
$n\times n$ matrices. A different generalization of the result of
Baker, Lawrence and Zorzitto was given by L. Sz\'ekelyhidi
\cite{ze1}, \cite{ze2}, \cite{ze3}. It involves an interesting
generalization of the class of bounded function on a group or
semigroup. For other
superstability results, we can see for example \cite{e}, \cite{b},  \cite{ger},  \cite{kim1}, \cite{kim2} and \cite{red}.\\ \\
  Let $G$ be a Hausdorff  locally compact group, $e$ its
identity element. Let $K$ be a compact subgroup of the group Mor(G)
of all mappings $k$ of $G$ onto itself that are either automorphisms
and homeomorphisms ($k \in K^{+}$), or antiautomorphisms and
homeomorphisms ($k \in K^{-}$). The action of $k \in K$ on $x \in G$
will be denoted by $k\cdot x$. Let $\mu$ be a complex bounded
measure on $G$ with compact support  (i.e, $\mu$ is an element of
the topological dual of the Banach spaces of continuous functions
vanishing at infinity on $G$). $\mu$ is assumed to be a
$K$-invariant measure that is,  $\int_{G}f(k\cdot
t)d\mu(t)=\int_{G}f( t)d\mu(t)$, for all $k\in K$
and for all continuous complex valued function $f$ on $G$.\\
The main purpose of this paper is to investigate the Hyers-Ulam
stability of the functional equations
 \begin{equation}\label{eq1}\int_{ G} \int_{ K}f(xtk\cdot
y)d\mu(t)dk=f(x)g(y), \; x, y \in G.\end{equation}
 Indeed we prove the superstability theorem of the
functional equation
\begin{equation} \label{eq2}\int_{ G} \int_{ K}f(xtk\cdot
y)d\mu(t)dk=f(x)f(y), \; x, y \in G.\end{equation} The functional
equation (\ref{eq1}) is a generalization of many functional
equations. The functional equation (\ref{eq2}) with
$\mu=\delta_{e}$: Dirac measure concentrated on the identity element
of $G$ reduce to $K$-spherical functions:
\begin{equation}\label{eq3} \int_{ K}f(xk\cdot
y)dk=f(x)f(y), \; x, y \in G.\
\end{equation}
The $K$-spherical functions and related equations has been widely
studied by H. Stetk\ae r see for example \cite{st8} and \cite{st9}.
The bounded solutions of $K-$spherical functions in an abelian group
are obtained by W. Chojnacki \cite{ch} and later by Badora \cite{ba}
while Stetk\ae r \cite{st1}, \cite{st2} studied unbounded solutions.
In \cite{sh} H. Shin'ya described all continuous solutions of
(\ref{eq3}) for abelian group. The functional equation (\ref{eq2})
is considered in \cite{el1}, \cite{el2} and \cite{ba}. The
functional equation
\begin{equation}\label{4}
\int_{ K}f(xk\cdot y)dk=f(x)g(y), \; x, y \in G
\end{equation} has been examined in special cases by many mathematicians. These
cases for example include the cosine equation or d'Alembert's
functional equations (cf. \cite{acz},\cite{st3}, \cite{st4}...)
\begin{equation}\label{5}
    f(x+y)+f(x-y)=2f(x)f(y),x,y\in G,
\end{equation}
Wilson's functional equation
\begin{equation}\label{6}
     f(x+y)+f(x-y)=2f(x)g(y),x,y\in G,
\end{equation}where $K=\{Id, -Id\}$ and the Cauchy's equation
\begin{equation}\label{7}
     f(x+y)=f(x)f(y),x,y\in G,
\end{equation} with $K=\{\ Id\}$.  \\ During the last three decades a number of papers and research
monographs have been published on various generalizations and
applications of the generalized Hyers-Ulam stability of a special
case of the functional equation (\ref{eq3}) and it's generalization
(\ref{4}). In \cite{baa2} R. Badora obtained the Hyers-Ulam
stability of equation (\ref{4})), where $G$ is abelian and
$K\subseteq Aut(G)$: the group of automorphisms of $G$. \\We note
here that the results of R. Badora \cite{baa2} are also corrects in
the case where $G$ is not necessarily abelian and $K\subseteq
Aut(G)$.  Other results of stability of functional equations related
to $K$-spherical functions where studied in \cite{badora1},
\cite{badora2}, \cite{badora3}, \cite{cha1} and \cite{cha2}.\\
In \cite{bbb} B. Bouikhalene and E. Elqorachi obtained  the
Hyers-Ulam stability of the functional equation (\ref{eq1}),
 provided that the continuous function $f$ satisfies the
   Kannappan type condition:
   $$\int_{G}\int_{G}f(ztxsy)d\mu(t)d\mu(t)=\int_{G}\int_{G}f(ztysx)d\mu(t)d\mu(t)$$
   for all $x,y,z\in G$.
  The purpose of this
paper is to remove this restriction.
\\ \\ Throughout this paper, $G$ is a locally compact group (not necessarily abelian) $K$ is a compact subgroup of morphisms of
$G$ and $\mu$ is a complex measure with copmact support and which is
$K$-invariant.
\section{Hyers Ulam stability of equation (\ref{eq1})} In this
section, we will investigate the Hyers Ulam stability of equation
(\ref{eq1}). The following lemma will be helpful in the sequel.
\begin{lem}Let $f$: $G\longrightarrow \mathbb{C}$ be a continuous function. Let $\mu$ be a complex measure with compact support and which is $K$-invariant. Then
\begin{equation}\label{8}
\int_{G}\int_{K}\int_{K}\int_{G}f(zth\cdot(k\cdot
ysx))d\mu(t)dhdkd\mu(s)+\int_{G}\int_{K}\int_{K}\int_{G}f(zth\cdot(xsk\cdot
y))d\mu(t)dhd\mu(s)dk\end{equation}$$=\int_{G}\int_{K}\int_{K}\int_{G}f(ztk\cdot
ysh\cdot
x)d\mu(t)dhdkd\mu(s)+\int_{G}\int_{K}\int_{K}\int_{G}f(zth\cdot
xsk\cdot y)d\mu(t)dhdkd\mu(s)$$ for all $x,y,z\in G.$

\end{lem}
\begin{proof}$$\int_{G}\int_{K}\int_{K}\int_{G}f(zth\cdot(k\cdot
ysx))d\mu(t)dhdkd\mu(s)+\int_{G}\int_{K}\int_{K}\int_{G}f(zth\cdot(xsk\cdot
y))d\mu(t)dhd\mu(s)dk$$
\begin{equation}\label{9}=
\int_{G}\int_{K^{+}}\int_{K}\int_{G}f(zthk\cdot yh\cdot sh\cdot
x)d\mu(t)dhdkd\mu(s)+\int_{G}\int_{K^{-}}\int_{K}\int_{G}f(zth\cdot
xh\cdot  shk\cdot y)d\mu(t)dhd\mu(s)dk
\end{equation}
$$+\int_{G}\int_{K^{+}}\int_{K}\int_{G}f(zth\cdot
xh\cdot shk\cdot
y)d\mu(t)dhdkd\mu(s)+\int_{G}\int_{K^{-}}\int_{K}\int_{G}f(zthk\cdot
yh\cdot sh\cdot x)d\mu(t)dhd\mu(s)dk.$$ Since $\mu$ is $K$-invariant
and the  Haar measure $dk$ is invariant, then we get
\begin{equation}\label{10}
\int_{G}\int_{K^{+}}\int_{K}\int_{G}f(zthk\cdot yh\cdot sh\cdot
x)d\mu(t)dhdkd\mu(s)=\int_{G}\int_{K^{+}}\int_{K}\int_{G}f(ztk\cdot
y sh\cdot x)d\mu(t)dhdkd\mu(s)\end{equation}

$$\int_{G}\int_{K^{-}}\int_{K}\int_{G}f(zth\cdot
xh\cdot  shk\cdot
y))d\mu(t)dhd\mu(s)dk=\int_{G}\int_{K^{-}}\int_{K}\int_{G}f(zth\cdot
x sk\cdot y)d\mu(t)dhd\mu(s)dk$$
$$\int_{G}\int_{K^{+}}\int_{K}\int_{G}f(zth\cdot
xh\cdot shk\cdot
y)d\mu(t)dhdkd\mu(s)=\int_{G}\int_{K^{+}}\int_{K}\int_{G}f(zth\cdot
x sk\cdot y)d\mu(t)dhdkd\mu(s)$$
$$\int_{G}\int_{K^{-}}\int_{K}\int_{G}f(zthk\cdot
yh\cdot sh\cdot
x)d\mu(t)dhd\mu(s)dk=\int_{G}\int_{K^{-}}\int_{K}\int_{G}f(ztk\cdot
y sh\cdot x)d\mu(t)dhd\mu(s)dk.$$ Now, by using (\ref{9}) and
(\ref{10}), we obtain
$$\int_{G}\int_{K}\int_{K}\int_{G}f(zth\cdot(k\cdot
ysx))d\mu(t)dhdkd\mu(s)+\int_{G}\int_{K}\int_{K}\int_{G}f(zth\cdot(xsk\cdot
y))d\mu(t)dhd\mu(s)dk$$
$$=\int_{G}\int_{K^{+}}\int_{K}\int_{G}f(ztk\cdot
y sh\cdot
x)d\mu(t)dhdkd\mu(s)+\int_{G}\int_{K^{-}}\int_{K}\int_{G}f(ztk\cdot
y sh\cdot x))d\mu(t)dhd\mu(s)dk$$
$$+\int_{G}\int_{K^{-}}\int_{K}\int_{G}f(zth\cdot
x sk\cdot
y)d\mu(t)dhd\mu(s)dk+\int_{G}\int_{K^{+}}\int_{K}\int_{G}f(zth\cdot
x sk\cdot y)d\mu(t)dhdkd\mu(s)$$
$$=\int_{G}\int_{K}\int_{K}\int_{G}f(ztk\cdot
ysh\cdot
x)d\mu(t)dhdkd\mu(s)+\int_{G}\int_{K}\int_{K}\int_{G}f(zth\cdot
xsk\cdot y)d\mu(t)dhdkd\mu(s),$$ which gives equation (\ref{8}).
\end{proof}
The main result of this section
\begin{thm} Let $\delta>0$. Suppose that the continuous functions $f,g$:
 $G\longrightarrow \mathbb{C}$ satisfy the inequality
\begin{equation}\label{11} |\int_{K}\int_{ G} f(xtk\cdot y)d\mu(t)dk-f(x)g(y)|<\delta
\end{equation} for all $x, y \in G.$  Then, \\
i) $f,g$ are bounded or\\ii) $f$ is unbounded and $g$ satisfies the
functional equation
\begin{equation}\label{12}
    \int_{K}\int_{ G} g(xtk\cdot y)d\mu(t)dk+\int_{K}\int_{ G} g(k\cdot
    ytx)d\mu(t)dk=2g(x)g(y),\;x,y\in G\end{equation}
    or iii) $g$ is unbounded, $f$ satisfies (\ref{eq1}) (if $f\neq 0$, then $g$ satisfies equation
    (\ref{12})).
\end{thm}
\begin{proof}Assume that $f,g$ are continuous and satisfy inequality (\ref{11}). In the first case, we suppose that $f$ is unbounded.
Then from (\ref{11}) we get
$$|\int_{G}\int_{K}\int_{K}\int_{G}f(zth\cdot(xsk\cdot
y))d\mu(t)dhdkd\mu(s)-f(z)\int_{G}\int_{K}g(xsk\cdot
y))dkd\mu(s)|$$$$\leq\int_{G}\int_{K}|\int_{K}\int_{G}f(zth\cdot(xsk\cdot
y))d\mu(t)dh-f(z)g(xsk\cdot y)|d|\mu|(s)dk\leq\delta \|\mu\|$$
$$|\int_{G}\int_{K}\int_{K}\int_{G}f(zth\cdot(k\cdot
ysx))d\mu(t)dhdkd\mu(s)-f(z)\int_{G}\int_{K}g(k\cdot
ysx))dkd\mu(s)|\leq \delta \| \mu\|$$
$$|\int_{G}\int_{K}\int_{K}\int_{G}f(zth\cdot xsk\cdot
y)d\mu(t)dhdkd\mu(s)-\int_{G}\int_{K}f(zth\cdot
x))dhd\mu(t)g(y)|\leq
 \delta
\| \mu\|$$
$$|\int_{G}\int_{K}\int_{K}\int_{G}f(zth\cdot ysk\cdot
x)d\mu(t)dhdkd\mu(s)-\int_{G}\int_{K}f(zth\cdot y))dhd\mu(t)g(x)
|\leq\delta\| \mu\|.$$So, by using lemma 2.1, the triangle
inequality, we obtain
$$|f(z)||2g(x)g(y)-\int_{G}\int_{K}g(xtk\cdot
y)d\mu(t)dk-\int_{G}\int_{K}g(k\cdot ytx)d\mu(t)dk|$$
$$\leq |\int_{G}\int_{K}\int_{G}\int_{K}f(zsh\cdot(xtk\cdot y))d\mu(t)d\mu(s)dhdk
-f(z)\int_{G}\int_{K}g(xtk\cdot y)d\mu(t)dk|$$
$$+|\int_{G}\int_{K}\int_{G}\int_{K}f(zsh\cdot(k\cdot ytx))d\mu(t)d\mu(s)dhdk
-f(z)\int_{G}\int_{K}g(k\cdot ytx)d\mu(t)dk|$$
$$+|\int_{G}\int_{K}\int_{G}\int_{K}f(zsh\cdot xtk\cdot y)d\mu(t)d\mu(s)dhdk
-\int_{G}\int_{K}f(zsh\cdot x)d\mu(s)dh g(y)|$$
$$+|\int_{G}\int_{K}\int_{G}\int_{K}f(zsh\cdot ytk\cdot x)d\mu(t)d\mu(s)dhdk
-\int_{G}\int_{K}f(zsk\cdot y)d\mu(s)dk g(x)|$$
$$+|g(y)||\int_{G}\int_{K}f(zth\cdot x)d\mu(t)dh
-f(z) g(x)|+|g(x)||\int_{G}\int_{K}f(zth\cdot y)d\mu(t)dh -f(z)
g(y)|.$$Therefore, $$|f(z)||2g(x)g(y)-\int_{G}\int_{K}g(xtk\cdot
y)d\mu(t)dk-\int_{G}\int_{K}g(k\cdot ytx)d\mu(t)dk|$$
$$\leq \delta
\| \mu\|+\delta \| \mu\|+\delta \| \mu\|+\delta \|
\mu\|+|g(y)|\delta+|g(x)|\delta.
$$Since $f$ is assumed to be unbounded, then we get
$$2g(x)g(y)-\int_{G}\int_{K}g(xtk\cdot
y)d\mu(t)dk-\int_{G}\int_{K}g(k\cdot ytx)d\mu(t)dk=0$$ for all
$x,y\in G.$ This proves the  case ii). Now, assume that $g$ is
unbounded. It's easily verified that $f=0$ satisfies equation
(\ref{eq1}). For latter, we suppose that $f\neq 0$. From inequality
(\ref{11}) and the triangle inequality, we conclude that $f$ is also
unbounded, then from the case ii) the function $g$ satisfies
equation (\ref{12}). For all $x,y,z\in G$, we
have$$|g(z)||\int_{G}\int_{K}f(xtk\cdot y)d\mu(t)dk-f(x)g(y)|$$
$$\leq |\int_{G}\int_{K}\int_{G}\int_{K}f(xtk\cdot ysh\cdot z)d\mu(t)d\mu(s)dhdk
-\int_{G}\int_{K}f(xtk\cdot y)d\mu(t)dkg(z)|$$
$$+ |\int_{G}\int_{K}\int_{G}\int_{K}f(xth\cdot(ysk\cdot z))d\mu(t)d\mu(s)dhdk
-f(x)\int_{G}\int_{K}g(ysk\cdot z)d\mu(s)dk|$$
$$+ |\int_{G}\int_{K}\int_{G}\int_{K}f(xth\cdot(k\cdot zsy))d\mu(t)d\mu(s)dhdk
-f(x)\int_{G}\int_{K}g(k\cdot zsy)d\mu(s)dk|$$
$$+ |\int_{G}\int_{K}\int_{G}\int_{K}f(xtk\cdot zsh\cdot y)d\mu(t)d\mu(s)dhdk
-\int_{G}\int_{K}f(xtk\cdot z)d\mu(t)dkg(y)|$$
$$+|g(y)||\int_{G}\int_{K}f(xtk\cdot z))d\mu(t)dk
-f(x)g(z)|$$
$$+|f(x)||\int_{G}\int_{K}g(ysk\cdot z)d\mu(s)dk+\int_{G}\int_{K}g(k\cdot z sy)d\mu(s)dk
-2g(y)g(z)|$$
$$\leq 4\delta \| \mu\| +|g(y)|\delta+|f(x)|\times 0=4\delta \| \mu\| +|g(y)|\delta.$$Since $g$ is
unbounded, then $f$ satisfies equation (\ref{eq1}). This completes
the proof. \end{proof} By using the above result, we get the
following corollary.\begin{cor}(Superstability of equation
(\ref{eq2})) Let $\delta>0$. If a continuous function $f$:
$G\longrightarrow \mathbb{C}$ satisfies the inequality

\begin{equation}\label{20}
    |\int_{G}\int_{K}f(xtk\cdot y)d\mu(t)dk-
f(x)f(y)|\leq\delta, \; x,y\in G.
\end{equation} Then either
\begin{equation}\label{21}
    |f(x)|\leq\frac{\|\mu\|+\sqrt{\|\mu\|^{2}+4\delta}}{2},x\in
    G
\end{equation} or
\begin{equation}\label{22}
\int_{G}\int_{K}f(xtk\cdot y)d\mu(t)dk =f(x)f(y),x,y\in G.
\end{equation}
\end{cor}
\begin{cor}(Superstability of the calssical d'Alembert's functional
equation)  Let $\mu=\delta_e$. Let $\delta>0$. Let $\sigma$:
$G\longrightarrow G$ be an involution of $G$ (
$\sigma(xy)=\sigma(y)\sigma(x)$ and $\sigma(\sigma(x))=x$ for all
$x,y\in G)$. Let $K=\{I,\sigma\}$.  If a function $f$:
$G\longrightarrow \mathbb{C}$ satisfies the inequality

\begin{equation}\label{40}
    |f(xy)+f(x\sigma(y))-
2f(x)f(y)|\leq\delta, \; x,y\in G.
\end{equation} Then either
\begin{equation}\label{41}
    |f(x)|\leq\frac{1+\sqrt{1+2\delta}}{2},x\in
    G
\end{equation} or
\begin{equation}\label{42}
f(xy)+f(x\sigma(y))= 2f(x)f(y),x,y\in G.
\end{equation}
\end{cor}
The following general corollary holds on any group and for
$K\subseteq Mor(G)$. It's a generalization of the result obtained by
Badora in \cite{baa2}.
\begin{cor} Let $\mu=\delta_e$. Let $\delta>0$. Suppose that the continuous function $f$:
 $G\longrightarrow \mathbb{C}$ satisfy the inequality
\begin{equation}\label{23} |\int_{K} f(xk\cdot y)dk-f(x)g(y)|<\delta, \;
x, y \in G.\end{equation}  Then, \\
i) $f,g$ are bounded or\\ii) $f$ is unbounded and $g$ satisfies the
functional equation
\begin{equation}\label{24}
    \int_{K} g(xk\cdot y)dk+\int_{K} g(k\cdot
    yx)dk=2g(x)g(y),\;x,y\in G\end{equation}
    or iii) $g$ is unbounded, $f$ satisfies (\ref{4}) (if $f\neq 0$, then $g$ satisfies equation
    (\ref{24})).
\end{cor}
\begin{cor}\cite{baa2} Let $\mu=\delta_e$, $K\subseteq Aut(G).$   Let $\delta>0$. Suppose that the continuous functions $f,g$:
 $G\longrightarrow \mathbb{C}$ satisfy the inequality
\begin{equation}\label{25} |\int_{K} f(xk\cdot y)dk-f(x)g(y)|<\delta, \;
x, y \in G.\end{equation}  Then, \\
i) $f,g$ are bounded or\\ii) $f$ is unbounded and $g$ satisfies the
functional equation
\begin{equation}\label{26}
    \int_{K} g(xk\cdot y)dk=g(x)g(y),\;x,y\in G\end{equation}
    or iii) $g$ is unbounded, $f$ satisfies (\ref{4}) (if $f\neq 0$, then $g$ satisfies equation
    (\ref{26})).
\end{cor}
\begin{cor} Let $\delta>0$. Let $K=\{Id, \sigma\}$, where $\sigma$ is an involution of $G$.
$\mu$ is a complex measure with compact support and which is $\sigma$-invariant.
 Suppose that the continuous functions
$f,g$ :
 $G\longrightarrow \mathbb{C}$ satisfy the inequality
\begin{equation}\label{27} |\int_{ G} f(xty)d\mu(t)+\int_{ G} f(xt\sigma(y))d\mu(t)-2f(x)g(y)|<\delta, \;
x, y \in G.\end{equation}  Then, \\
i) $f,g$ are bounded or\\ii) $f$ is unbounded and $g$ satisfies the
functional equation
\begin{equation}\label{27}
    \int_{ G} g(xt y)d\mu(t)+\int_{ G} g(yt x)d\mu(t)+\int_{ G} g(\sigma(y)tx)d\mu(t)+\int_{ G} g(xt \sigma(y))d\mu(t)=4g(x)g(y)\end{equation}
    or iii) $g$ is unbounded, $f$ satisfies
\begin{equation}\label{40}
    \int_{ G} f(xty)d\mu(t)+\int_{ G} f(xt\sigma(y))d\mu(t)=2f(x)g(y),x,y\in G\end{equation}
    (if $f\neq 0$, then $g$ satisfies equation
    (\ref{27})).
\end{cor}
The following corollary is a generalization of the result obtained
by E. Elqorachi and M. Akkouchi in \cite{e} under the condition that
$f$ satisfies the  Kannappan type condition or  $\mu$ is a
generalized Gelfand measure.
\begin{cor} Let $\delta>0$. Let $K=\{Id, \sigma\}$, where $\sigma$ is an involution of $G$. $\mu$ is a complex measure with compact support and which is $\sigma$-invariant.
 Suppose that the continuous functions
$f,g$ :
 $G\longrightarrow \mathbb{C}$ satisfy the inequality
\begin{equation}\label{50} |\int_{ G} f(xty)d\mu(t)+\int_{ G} f(xt\sigma(y))d\mu(t)-2f(x)f(y)|<\delta, \;
x, y \in G.\end{equation}   Then either
\begin{equation}\label{51}
    |f(x)|\leq\frac{\|\mu\|+\sqrt{\|\mu\|^{2}+2\delta}}{2},x\in
    G
\end{equation} or
\begin{equation}\label{52}
    \int_{ G} f(xty)d\mu(t)+\int_{ G} f(xt\sigma(y))d\mu(t)=2f(x)f(y),x,y\in G\end{equation}

\end{cor}

\textsc{ Bouikhalene Belaid, }  \textsc{Polydisciplinary Faculty,
University Sultan Moulay Slimane, Beni-Mellal, Morocco},\\ \textsc{E-mail} : bbouikhalene@yahoo.fr\\
\\\textsc{ Elqorachi Elhoucien, Department of Mathematics,
Faculty of Sciences, University  Ibn Zohr, Agadir, Morocco, \\
E-mail}: elqorachi@hotamail.com\\


\begin{thebibliography}{99.}%
\bibitem{acz} J. Acz\'el and J. Dhombres, Functional equations in several
variables. With applications to mathematics, information theory and
to the natural and social sciences. Encyclopedia of Mathematics and
its Applications, 31. Cambridge University Press, Cambridge, 1989.
\bibitem{badora1} R. Badora, Stability Properties of Some Functional
Equations. In: Themistocles Rassias, Janusz Brzdek (ed.) Functional
Equations in Mathematical Analysis, pp.3-13. Springer Optimization
and Its Applications 52, 2011.
\bibitem{badora2} R. Badora, On the stability of a functional equation for generalized trigonometric functions.
In: Th.M. Rassias (ed.) Functional Equations and Inequalities,
pp.1-5. Kluwer Academic Publishers, 2000.
\bibitem{badora3} R. Badora,  On the
stability of some functional equations. In: Report of Meeting, 10th
International Conference on Functional Equations and Inequalities
(September 11-17, 2005, Be¸dlewo, Poland), p.130. Ann. Acad. Paed.
Cracoviensis Studia Math. 5(2006)
\bibitem{ba} R. Badora, On a joint generalization of Cauchy's and d'Alembert's
functional equations. Aequationes Math. 43(1992), No. 1, 72–89.
\bibitem{baa2} Badora R., On Hyers-Ulam stability of Wilson's functional
equation, Aequationes Math. 60(2000), 211-218.
\bibitem{baa3}R. Badora, "Note on
the superstability of the Cauchy functional equation," Publicationes
Mathematicae Debrecen, vol. 57(2000), No. 3-4,  421-424.
\bibitem{j2} Baker J. A., The
stability of the cosine equation, Proc. Amer. Math. Soc. 80(1980),
411-416.
\bibitem{j1} Baker J., Lawrence J. and Zorzitto F., The stability of
the equation f(x + y) = f(x)f(y), Proc. Amer. Math. Soc. 74(1979),
242-246.
 \bibitem{bbb} B. Bouikhalene and E. Elqorachi, On Stetk\ae r type functional equations and Hyers–Ulam
stability, Publicationes Math. 69(2006), No.1-2(6).
\bibitem{2} B. Bouikhalene, E. Elqorachi and Th. M. Rassias, On the Hyers-Ulam stability of
approximately Pexider mappings. Math. Inequal. Appl., 11(2008),
805-818.
\bibitem{b} B. Bouikhalene, E. Elqorachi, and J. M. Rassias, "The superstability
of d'Alembert's functional equation on the Heisenberg group,"
Applied Mathematics Letters, 23(2000), No.1, 105-109.
\bibitem{cad1}L. C\v{a}dariu and V. Radu, Fixed points and the stability of Jensens
functional equation, Journal of Inequalities in Pure and Applied
Mathematics,  4(2003), No. 3 article 4.
 \bibitem{cha1} A.
Charifi, B. Bouikhalene and E. Elqorachi, Hyers-Ulam-Rassias
stability of a generalized Pexider functional equation, Banach J.
Math. Anal., 1(2007), 176-185.
\bibitem{cha2} A. Charifi, B. Bouikhalene, E.
Elqorachi and A. Redouani, Hyers-Ulam-Rassias Stability of a
generalized Jensen functional equation, Australian J. Math. Anal.
Appli. 19(2009), 1-16.
\bibitem{ch} W. Chojnacki, On some functional equation generalizing Cauchy's and
d'Alembert's functional equations. Colloq. Math. 55(1988), No. 1,
169-178.
\bibitem{e} Elqorachi E. and Akkouchi M., The superstability of the generalized
d'Alembert functional equation, Georgian Math. J. 10(2003), 503-508.
\bibitem{el1} E. Elqorachi, M.
Akkouchi, A. Bakali, and B. Bouikhalene, Badora's equation on
non-abelian locally compact groups. Georgian Math. J. 11(2004), No.
3, 449-466.
\bibitem{el2} E. Elqorachi, M.
Akkouchi and B. Bouikhalene.  Functional Equation and
$\mu$-Spherical Functions. Georgian Math. J. 15 (2008), No. 1, 1-20
\bibitem{5} Forti, G. L.,
Hyers-Ulam stability of functional equations in several variables.
Aequationes Math. {50} (1995), 143-190.
\bibitem{6} Gajda, Z., On stability of
additive mappings, {\em{Internat. J. Math. Sci.}} {\bf{14}} (1991),
431-434.
\bibitem{7} G\v{a}vruta, P. A generalization of the Hyers-Ulam-Rassias stability
of approximately additive mappings. J. Math. Anal. Appl. 184 (1994),
431-436.
\bibitem{ger} Ger R., Superstability is not natural, Rocznik
Nauk.-Dydakt. Prace Mat. 159 (1993), No. 13, 109-123.
\bibitem{se} Ger R. and \v{S}emrl P., The stability of the exponential equation,
Proc. Amer. Soc. V. 124 (1996), 779-787.
\bibitem{8}
 Hyers, D. H., On
the stability of the linear functional equation, {\em{Proc. Nat.
Acad. Sci. U. S. A.}} {{27}}(1941), 222-224.
 \bibitem{9}
  Hyers, D. H.  and
 Rassias, Th. M., {\em{Approximate homomorphisms}},
{\em{Aequationes Math.,}} {{44}}(1992), 125-153.
\bibitem{10}
  Hyers, D. H.,  Isac, G. I.  and
 Rassias, Th. M., Stability of Functional Equations
in Several Variables, Birkh\"{a}user, Basel, 1998.
\bibitem{12}
  Jung, S.-M., Hyers-Ulam-Rassias Stability of
Functional Equations in Mathematical Analysis, {\em{Hadronic Press,
Inc., Palm Harbor, Florida,}} 2003.
\bibitem{13}
  Jung, S.-M.,  Stability of the quadratic equation of
Pexider type,
   {\em{Abh. Math. Sem. Univ. Hamburg,}} {70}(2000), 175-190.
\bibitem{15}
  Jung, S.-M.,  Sahoo, P. K.,  Stability of a functional
equation of Drygas, \textit{Aequationes Math.} {64}(2002),
 No. 3, 263-273.
\bibitem{kim1} G. H. Kim, On the stability of trigonometric functional
equations, Advances in Difference Equations, vol. 2007, Article ID
90405, 10 pages.
\bibitem{kim2} G. H. Kim, "On the stability of the
Pexiderized trigonometric functional equation," Applied Mathematics
and Computation, 203(2008), No. 1,  99-105.
\bibitem{la} J. Lawrence, The stability of multiplicative semigroup
homomorphisms to real normed algebras, Aequationes Math. 28 (1985),
94-101.
\bibitem{mos2}M. S. Moslehian, The Jensen functional equation in
non-Archimedean normed spaces, J. Funct. Spaces Appl., 7 (2009),
13-24.
\bibitem{mos3} M. S. Moslehian and Gh. Sadeghi, Stability of linear
mappings in quasi-Banach modules, Math. Inequal. Appl., 11 (2008),
549-557.
\bibitem{mos4} A. Najati and M. B. Moghimi, Stability of a functional
equation deriving from quadratic and additive functions in
quasi-Banach spaces, J. Math. Anal. Appl., 337 (2008), 399-415.
\bibitem{mos5}A.
Najati, On the stability of a quartic functional equation, J. Math.
Anal. Appl., 340 (2008), 569-574.
\bibitem{mos6} A. Najati and C. Park,
Hyers-Ulam-Rassias stability of homomorphisms in quasi- Banach
algebras associated to the Pexiderized Cauchy functional equation,
J. Math. Anal. Appl., 335 (2007), 763-778.
\bibitem{mos9} M. M. Pourpasha, J. M. Rassias, R. Saadati and S. M.
Vaezpour, A fixed point approach to the stability of Pexider
quadratic functional equation with involution, J. Ineq. Appl. (2010)
Article ID 839639, doi:10.1155/2010/839639.
\bibitem{mos10} J. M. Rassias, On
approximation of approximately linear mappings by linear mappings,
J. Funct. Anal., 46(1982), 126-130.
\bibitem{mos11} J. M. Rassias, Solution of a
problem of Ulam, J. Approx. Theory., 57(1989), 268-273.
\bibitem{mos12}J. M.
Rassias, On the Ulam stability of mixed type mappings on restricted
domains, J. Math. Anal. Appl., 276 (2002), No.2, 747-762.
\bibitem{16}
 Rassias, Th. M., {\em{On the stability of linear
mapping in Banach spaces}}, {\em{Proc. Amer. Math. Soc.}}
{{72}}(1978), 297-300.
 \bibitem{17}
  Rassias, Th. M.,  The problem of S. M. Ulam for
approximately
  multiplicative mappings,
 {\em{J.  Math. Anal.
Appl.}}   {{246}}(2000), 352-378.
\bibitem{18}
  Rassias, Th. M., On the stability of the functional
equations and a problem of Ulam,
 {\em{Acta Applicandae Mathematicae.}}   {{62}}(2000), 23-130.
\bibitem{20}
  Rassias, Th. M. and Tabor J., Stability of
Mappings of Hyers-Ulam Type,
 {\em{Hardronic Press, Inc., Palm Harbor, Florida}}
1994.
\bibitem{sh} H. Shin'ya, Spherical matrix functions and Banach
representability for locally compact motion groups. Japan. J. Math.
(N.S.) 28(2002), No. 2, 163–201.
\bibitem{red} A. Redouani,
E. Elqorachi, M. Th. Rassias
 The superstability of d'Alembert's functional equation on
step $2$ nilpotent groups, Aequationes math., 74(2007), No. 3,
226-241.
\bibitem{st2} H. {Stetk\ae r} , Functional equations and
matrix-valued spherical functions. Aequationes Math. 69(2005), No.
3, 271-292.
\bibitem{st1}{H. {Stetk\ae r} } Functional Equations and Spherical
Functions. \textit{Preprint Series 1994 No. 18, Mathematisk
Institut, Aarhus University, Denmark pp. 1-18.}
\bibitem{st3} H. Stetk\ae r,
Functional equations on abelian groups with involution. Aequationes
Math. 54(1997), No. 1-2, 144-172.
\bibitem{st4} H. Stetk\ae r, Functional equations on groups, Word
Scientific, New Jersey, London, Singapore, Bejing, Shanghal, Hong
Kong, Taipei, Chennai, 2013. \bibitem{st8} H. Stetk\ae r,
d'Alembert's equation and spherical functions. Aequationes Math. 48
(1994), No. 2-3, 220-227. \bibitem{st9} H. Stetk\ae r, Wilson's
functional equations on groups. Aequationes Math. 49(1995), No. 3,
252-275.
\bibitem{ze1} Sz\'ekelyhidi L., On a theorem of Baker, Lawrence and Zorzitto,
Proc. Amer. Math. Soc. 84 (1982), 95-96.
\bibitem{ze2} Sz\'ekelyhidi L., The
stability of d'Alembert-type functional equations, Acta Sci. Math.
Szeged 44 (1982), 313-320.
  \bibitem{ze3} Sz\'ekelyhidi L., On a stability
theorem, C. R. Math. Acad. Sc. Canada 3 (1981), 253-255.
\bibitem{24} Ulam S. M., {\em{
A Collection of Mathematical Problems}}, {Interscience Publ. New
York,} 1961. Problems in Modern Mathematics, Wiley, New York 1964.


\end{thebibliography}
\end{document}